\documentclass[10pt]{article}
\usepackage{graphicx}
\usepackage{amsmath}
\usepackage{amsfonts}
\usepackage{amssymb}
\usepackage{amsthm}
\usepackage[shortlabels]{enumitem} 
\usepackage[cp1250]{inputenc}
\usepackage[a4paper, left=3cm, right=3cm, top=2.8cm, bottom=2.8cm, headsep=1.2cm]{geometry}

\usepackage{caption}
\usepackage{color}
\definecolor{sepia}{cmyk}{0, 0.83, 1, 0.70}

\usepackage[plainpages=false, pdfpagelabels, bookmarksnumbered, colorlinks=true, linkcolor=sepia, citecolor=sepia, unicode]{hyperref}

\usepackage{makecell} 
\usepackage{multirow}
\usepackage{rotating}

\newtheorem{theorem}{Theorem}

\newtheorem{lemma}[theorem]{Lemma}

\newtheorem{proposition}[theorem]{Proposition}

\theoremstyle{definition}
\newtheorem{definition}[theorem]{Definition}
\newtheorem{remark}[theorem]{Remark}
\newtheorem{example}[theorem]{Example}

\numberwithin{equation}{section} 
\numberwithin{theorem}{section}  
\numberwithin{figure}{section}   
\numberwithin{table}{section}    

\makeatletter
\renewenvironment{proof}[1][\proofname]
{\par
	\pushQED{$\blacksquare$} 
	\normalfont\topsep6\p@\@plus6\p@\relax
	\trivlist
	\item[\hskip\labelsep\bfseries#1\@addpunct{.}]
	\ignorespaces}
{\popQED \endtrivlist\@endpefalse}
\makeatother

\DeclareMathOperator*{\interior}{int}

\DeclareMathOperator*{\fix}{Fix}

\DeclareMathOperator*{\Argmin}{Argmin}
\DeclareMathOperator*{\argmin}{argmin}

\DeclareMathOperator*{\argmax}{argmax}
\DeclareMathOperator*{\prob}{Pr}
\DeclareMathOperator*{\prox}{Prox}
\DeclareMathOperator*{\Mod}{mod}

\begin{document}

\title{\vspace{-4em}\textbf{\Large Finitely Convergent Deterministic and Stochastic  Iterative  Methods for Solving Convex Feasibility Problems}}

\author{Victor I. Kolobov\thanks{Department of Computer Science, The Technion -- Israel Institute of Technology, 32000 Haifa, Israel, kolobov.victor@gmail.com}
\and Simeon Reich\thanks{Department of Mathematics, The Technion -- Israel Institute of Technology, 32000 Haifa, Israel, \protect\\sreich@technion.ac.il}
\and Rafa\l\ Zalas\thanks{Department of Mathematics, The Technion -- Israel Institute of Technology, 32000 Haifa, Israel, \protect\\rzalas@campus.technion.ac.il}}
\maketitle

\begin{abstract}
We propose finitely convergent methods for solving convex feasibility problems defined over a possibly infinite pool of constraints. Following other works in this area, we assume that the interior of the solution set is nonempty and that certain overrelaxation parameters form a divergent series. We combine our methods with a very general class of deterministic control sequences where, roughly speaking, we require that sooner or later we encounter a violated constraint if one exists. This requirement is satisfied, in particular, by the cyclic, repetitive and remotest set controls. Moreover, it is almost surely satisfied for random controls.

\end{abstract}

\section{Introduction}
The \textit{convex feasibility problem} (CFP) is one of the fundamental problems in optimization. It is mainly used for (but not limited to) modeling problems in which finding a solution satisfying a given list of constraints is satisfactory, whereas obtaining the optimal solution is not necessary. In this paper we consider the following variant of the CFP defined in a real Hilbert space $\mathcal H$: Find $x\in C\cap Q$ with $C:=\bigcap_{i\in I} C_i$. We assume that each one of the sets $C_i$, $i \in I$, as well as $Q$, is closed and convex and $I:=\{1,2,\ldots,m\}$ for some $m \in \mathbb N_+\cup\{\infty\}$.

Oftentimes it is convenient to represent the constraint set $C_i$ as the fixed point set of an operator $T_i\colon\mathcal H \to \mathcal H$, which satisfies certain conditions -- a variation of firm nonexpansivity. In particular, one can use the metric projection $P_{C_i}$, the proximal operator $\prox_{f_i}$ (when $C_i$ is the set of minimizers of a proper, l.s.c. and convex function $f_i$) or the subgradient projection $P_{f_i}$ (when $C_i$ is the sublevel set of a real-valued, l.s.c. and convex function $f_i$). Note, however, that in some cases $P_{f_i}$ may happen to be discontinuous (see \cite[Example 29.47]{BC17}). For this reason, in this paper we consider a class of operators called \emph{cutters}, which, by some authors, are also called \emph{firmly quasi-nonexpansive}; see \cite[Definition 4.1]{BC17}. A comprehensive overview of this class, citing many relevant references, can be found in \cite{Ceg12}. Here we only note that in view of \cite[Theorem 2.2.5]{Ceg12}, a cutter can be considered a generalization of a firmly nonexpansive mapping, provided it has a fixed point.

There is a great number of deterministic and stochastic Fej\'{e}r monotone algorithms designed for solving the CFP, which are governed by cutter operators; see, for example, \cite{BB96, BC01, BNPH15, BLT15, Ceg12, CC11, CRZ18, Com96, Com01, HLS19, Ned11} to name but a few. A typical result ensures the asymptotic convergence of the generated iterates. For example, this can be weak, norm or linear convergence, but also their stochastic versions. A potential drawback, when using these algorithms, is that the produced iterates may never be feasible, unlike the limit point to which they converge. An example of such a situation has recently been given in \cite[Theorem 7]{LTT20}.

The purpose of this paper is to propose a modification of the iterative method discussed in $\cite{BB96}$, so that the sequence of the generated iterates reaches the solution set $C\cap Q$ within a finite number of steps in the presence of the following constraint qualification: $\interior (C)\cap Q\neq\emptyset$. We emphasize here that our approach is not the only way of achieving finite convergence. Other results that discuss this topic can be found, for instance, in \cite{BD17, BDNP16, CH10, Fuk82, HC08, IM87, Pan15}.

To be more precise, we focus on a projected version of the framework considered in $\cite{BB96}$, which can be formulated as follows:
\begin{equation}\label{int:xk}
  x_0\in Q,\quad
  x_{k+1}:=
    P_Q\left(x_k+\alpha_{k}\sum_{i\in I_k(x_k)}
    \lambda_{i,k}(x_k)\Big(T_i(x_k)-x_k\Big)\right),
\end{equation}
where $\alpha_k\in (0,2]$ is the \emph{relaxation parameter}, the weights $\lambda_{i,k}(x_k)\in [0,1]$ satisfy $\sum_{i\in I_k} \lambda_{i,k}(x_k) = 1$ and $I_k(x_k) \subset I$ is a nonempty and finite set of indices. The idea of the above-mentioned modification is to extend each nonzero vector $(T_i(x_k)-x_k)$ from \eqref{int:xk} by a scalar $r_{[k]} / \varphi_i(x_k)$, where $\{r_k\}_{k=0}^\infty \subset (0,\infty)$ is a sequence of small \emph{overrelaxation} parameters, $[k]\in\{0,\ldots,k\}$ counts all the correction steps up to the $k$-th iterate and $\varphi_i(x_k)\in (0,\infty)$ satisfies a certain boundedness condition; see Theorem \ref{thm:main}.

We observed, that after such a modification, whenever we do a correction step ($x_{k+1}\neq x_k$), we move towards an interior point $z\in \interior(C)\cap Q$ under the assumption that the overrelaxation parameter $r_{[k]}$ is small enough. In particular, we show that
\begin{equation}\label{int:descent}
  \|x_{k+1}-z\|^2\leq \|x_k-z\|^2 - 2 M R r_{[k]},
\end{equation}
where $M>0$ is some constant and $B(z,2R)\subseteq C$. It is not difficult to see that by repeatedly applying the above inequality one could arrive at a contradiction knowing that correction steps happen a considerable number of times and $\sum_{k=0}^{\infty}r_k=\infty$. This simple argument suggests that eventually we should encounter an iterate $x_k\in C\cap Q$ for some $k$.

Special cases of our framework can be found in the literature; see, for example, \cite{CCP11, DePI88, IM86}, where $\varphi_i(x_k):=\|g_i(x_k)\|$ ($g_i(x_k)$ is a subgradient of a convex function $f_i$ describing the sublevel set $C_i$) and \cite{BWWX15, Cro04, Pol01}, where $\varphi_i(x_k):=1$. For a more detailed description of the above-mentioned works, see Table \ref{table}. Here we only note that the convergence analysis in \cite{CCP11, DePI88, IM86} differs significantly from the one presented in \cite{BWWX15, Cro04, Pol01}. In this paper we follow the path set in \cite{BWWX15, Cro04, Pol01}, where one can find weaker forms of inequality \eqref{int:descent}. Moreover, in our approach, by introducing general functionals $\varphi_i$, we show that this path can also be used for $\varphi_i(x) = \|g_i(x)\|$, as is the case in \cite{CCP11, DePI88, IM86}. In particular, we establish inequality \eqref{int:descent} for subgradient projection methods from \cite{CCP11, DePI88, IM86}, a property which was not known before.

Various strategies defining the control sequence $\{I_k\}_{k=0}^\infty$ are known in the literature, ranging from deterministic, where $I_k\colon \mathcal H \to 2^I\setminus \emptyset$, to random, where $I_k\colon \Omega \to 2^I\setminus \emptyset$ are i.i.d. defined in some probability space. In the former case, the most recognizable are almost cyclic, intermittent and repetitive (chaotic) control sequences; see \cite{BB96} but also Table \ref{table}. The common feature of all the above-mentioned deterministic control sequences is that
\begin{equation}\label{int:control}
  \#(\{k \geq 0 \colon  I_k(x)\cap I_+(x)\neq\emptyset \})=\infty
\end{equation}
for all $x \notin C$, where $I_+(x):=\{i\in I \colon x \notin C_i\}$. Moreover, one can show that the random controls almost surely satisfy condition \eqref{int:control} provided that the probability of performing a correction step is positive whenever the iterate is outside the set $C$. We note that the latter condition, which we formally write in \eqref{thm:main:stoch:PrAssumption}, was proposed by Polyak in \cite[Assumption 2]{Pol01}, see also \cite[Assumption 1]{CP01}. What we show in this paper is that condition \eqref{int:control} is actually a sufficient one for obtaining the finite convergence of the modified method \eqref{int:xk} discussed above.

The use of the counter $[k]$ instead of $k$ was originally suggested in \cite[Section 4.2]{Pol01} for a particular variant of our general framework (compare with Table \ref{table}). Note that there are certain situations, where $[k]=k$ or when $[k]$ can simply be omitted; see Remark \ref{rem:squareBrackets}. However, it is important to notice that in some cases, not using the counter $[k]$ may lead to lack of finite convergence. To illustrate this, we provide two counterexamples. In the relatively simple Example \ref{ex:NoFiniteConv}, we show that finite convergence does not occur for the alternating projection method with relaxation applied to two halfspaces in the plane. In Example \ref{ex:NoFiniteConv2}, which is more important, but also more technical, we assume that the control sequence is repetitive, $r_k\downarrow 0$ (monotonically) and $\sum_{k=1}^{\infty} r_k = \infty$. In particular, this answers a question raised in \cite{CCP11} related to the use of repetitive controls (called expanding therein), showing that in general \cite[Theorem 20]{CCP11} cannot hold without a certain technical assumption \cite[Condition 19]{CC11}. On the other hand, in view of our result, \cite[Condition 19]{CC11} is no longer necessary for repetitive controls, when combined with $[k]$.

The contribution of our paper can be summarized as follows: We develop a unified analysis for a large class of finitely convergent iterative methods, which employ certain overrelaxation parameters. We also introduce a very broad class of control sequences, which covers a wide range of deterministic iterative methods and further allows us to consider stochastic iterative methods. As a byproduct, we improve all the results from \cite{BWWX15, CCP11, Cro04, DePI88, IM86, Pol01}; see the last row in Table \ref{table}.

Our paper is organized as follows. In section \ref{sec:preliminaries} we provide basic facts related to cutters. In section \ref{sec:auxiliary} we establish two auxiliary lemmata which are the key tools in our analysis. In particular, Lemma \ref{lem:basic} laid the foundation for establishing inequality \eqref{int:descent}. In section \ref{sec:mainResultsDet} we discuss in detail condition \eqref{int:control} together with several examples. We present there our main result, namely Theorem \ref{thm:main}. In section \ref{sec:mainResultsStoch} we formulate a stochastic counterpart of Theorem \ref{thm:main} (Theorem \ref{thm:main:stoch}). In the \hyperref[sec:Appendix]{Appendix} we provide two counterexamples showing that omitting the counter $[k]$ may cause lack of finite convergence.

\begin{table}[h]
\centering

\resizebox{1.0\linewidth}{!}{
\begin{tabular}{|c | c| c| c| c| c|}
\hline
\thead{Result} &
\thead{Constraints\\and Operators} &
\thead{Constraints\\Qualification}&
\thead{$\varphi_i(x)$}&
\thead{Overrelaxation\\and Relaxation}&
\thead{Control\\Sequence}
\\

\hline 
\makecell{Iusem, Moledo\\\cite[Theorem 1]{IM86}} &
\makecell{$C_i=S(f_i,0)$\\$T_i = P_{f_i}$\\$m\in\mathbb N_+$\\$Q=\mathbb R^n$} &
$\max\limits_{i\in I}f_i(x)<0$ &
\makecell{$\varphi_i(x)=\|g_i(x)\|$\\$g_i(x)\in\partial f_i(x)$\\for $x \notin C_i$} &
\makecell{$r_k \downarrow 0$, $\sum\limits_{k=0}^{\infty}r_k=\infty$\\$\alpha_k\in[\varepsilon,2-\varepsilon]$} &
\makecell{$I_k=I$}
\\

\hline 
\makecell{De Pierro, Iusem\\\cite[Theorem 1]{DePI88}} &
as in \cite{IM86}&
as in \cite{IM86}&
as in \cite{IM86}&
as in \cite{IM86}&
\makecell{almost cyclic}
\\

\hline 
\makecell{Censor, Chen, \\Pajoohesh\\\cite[Theorem 20]{CCP11}} &
as in \cite{IM86}&
as in \cite{IM86}&
as in \cite{IM86}&
as in \cite{IM86}&
\makecell{repetitive\\ (single valued)\\$+$\\ \cite[Condition 19]{CCP11}}
\\

\hline 
\makecell{Polyak\\ \cite[Theorem 1]{Pol01}\\$+$\\\cite[Section 4.2]{Pol01}} &
\makecell{$C_i=S(f_i,0)$\\$T_i = P_{f_i}$\\$m\in\mathbb N_+\cup\{\infty\}$\\ $Q\subset\mathbb R^n$}&
$\interior(C)\cap Q \neq \emptyset$&
$\varphi_i(x)=1$&
\makecell{$r_k = r$ or \\
$r_k \to 0$, $\sum\limits_{k=0}^{\infty}r_k^2=\infty$\\(combined with $[k]$)
\\$\alpha_k=1$}&
\makecell{random\\ (single valued)}
\\

\hline 
\makecell{Crombez\\ \cite[Theorem 2.7]{Cro04}} &
\makecell{$C_i=\fix T_i$\\$T_i$ - cutter\\$m\in\mathbb N_+$\\$Q=\mathbb R^n$\\ }&
$\interior(C) \neq \emptyset$&
$\varphi_i(x)=1$&
\makecell{$r_k = r$\\$\alpha_k=1$}&
\makecell{$I_k(x)\subset I_+(x)$}
\\

\hline 
\makecell{Bauschke, Wang\\Wang, Xu\\\cite[Theorem 3.1]{BWWX15}} &
\makecell{$C=\fix T$\\$T$ - cutter\\$m = 1$\\$Q\subset\mathbb R^n$}&
$\interior(C\cap Q) \neq \emptyset$&
$\varphi_i(x)=1$&
\makecell{$r_k \to 0$\\$\sum\limits_{k=0}^{\infty}\alpha_k r_k=\infty$\\$\alpha_k \in (0,2]$}&
\makecell{x}
\\

\hline 
\makecell{Bauschke, Wang\\Wang, Xu\\\cite[Theorem 3.2]{BWWX15}} &
\makecell{as above}&
$\interior(C)\cap Q \neq \emptyset$&
as above&
\makecell{$r_k \to 0$\\$\sum\limits_{k=0}^{\infty}\alpha_k(2-\alpha_k) r_k^2=\infty$\\$\alpha_k \in (0,2]$}&
\makecell{x}
\\

\hline \hline
\makecell{\textbf{Current Paper}\\ Theorems \ref{thm:main} and \ref{thm:main:stoch}\\(see also Ex. \ref{ex:phi2})} &
\makecell{$C_i=\fix T_i$\\$T_i$ - cutter\\$m\in\mathbb N_+\cup\{\infty\}$\\ $Q\subset\mathcal H$}&
$\interior(C)\cap Q \neq \emptyset$&
\makecell{$\delta\leq\varphi_i(x)\leq\Delta$\\on bounded sets\\$\delta,\Delta\in(0,\infty)$\\(all above)}&
\makecell{$r_k = r$ or \\
$r_k \to 0$, $\sum\limits_{k=0}^{\infty}\alpha_kr_k=\infty$\\(combined with $[k]$)
\\$\alpha_k\in(0,2]$}&
\makecell{well matched\\or\\random}
\\
\hline
\end{tabular}
}
\caption{\footnotesize The symbols $P_{f_i}$ and $S(f_i,0)$ refer to the subgradient projection and the sublevel set, respectively; see Example \ref{ex:subProj}. A constant overrelaxation parameter $r\leq R$ is chosen so that $B(z,2R)\subset C$ for some $z\in Q$ and applies only to $\varphi_i(x)=1$. The result of Crombez featuring string averaging is reduced to singleton strings (\cite[$n(t)=1$]{Cro04}).
    }
\label{table}
\end{table}

\section{Preliminaries} \label{sec:preliminaries}
Let $\mathcal H$ be a real Hilbert space with inner product $\langle \cdot, \cdot \rangle$ and induced norm $\|\cdot\|$.

\begin{definition}\label{def:cutter}
  Let $T\colon \mathcal H \to \mathcal H$ be an operator with $\fix T:=\{z\in\mathcal H \colon T(z)=z\}\neq \emptyset$. We say that $T$ is a \emph{cutter} if $\langle x-T(x), z-T(x)\rangle \leq 0$ for all $x\in\mathcal H$ and $z\in \fix T$.
\end{definition}

\begin{example}[Metric Projection]
  Let $C\subseteq \mathcal H$ be nonempty, closed and convex. The \emph{metric projection} $P_C(x):=\argmin_{z\in C}\|z-x\|$  is a cutter and  $\fix P_C = C$; see, for example, \cite[Theorem 1.2.4]{Ceg12}.
\end{example}

\begin{example}[Proximal Operator]
  Let $f\colon\mathcal H \to \mathbb R \cup\{+\infty\}$ be a lower semicontinuous and convex function. The \emph{proximal operator}  $\prox\nolimits_f (x) := \argmin_{y\in \mathcal H} ( f(y)+\frac 1 2 \|y-x\|^2)$  is firmly nonexpansive and $\fix (\prox_f)=\Argmin_{x\in\mathcal H} f(x)$; see \cite[Propositions 12.28 and 12.29]{BC17}. Thus if $f$ has at least one minimizer, then $\prox_f$ is a cutter; see \cite[Theorem 2.2.5]{Ceg12}.
\end{example}

\begin{example}[Subgradient Projection]\label{ex:subProj}
  Let $f\colon \mathcal{H}\to \mathbb{R}$ be a lower semicontinuous and convex function with nonempty sublevel set $S(f,0):=\{x\in \mathcal{H}\colon f(x)\leq 0\}\neq \emptyset $. For each $x\in \mathcal{H}$, let  $g(x)$  be a chosen subgradient from the subdifferential set $\partial f(x):=\{g\in \mathcal{H}\colon f(y)\geq f(x)+\langle g,y-x\rangle \text{, for all }y\in \mathcal{H}\}$,  which, by \cite[Proposition 16.27]{BC17}, is nonempty.  Note here that we apply \cite[Proposition 16.27]{BC17} to a real-valued functional $f$.  The  \emph{subgradient projection}
  \begin{equation}\label{ex:subProj:eq}
    P_{f}(x):=
    \begin{cases}
      x-\frac{f(x)}{\|g(x)\|^2}g(x), & \mbox{if } f(x)>0 \\
      x, & \mbox{otherwise}
    \end{cases}
  \end{equation}
  is a cutter  and $\fix P_{f}=S(f,0)$; see, for example, \cite[Corollary 4.2.6]{Ceg12}.
\end{example}

Note that both the proximal operator and the subgradient projection extend the notion of the metric projection; see \cite[Examples 12.25 and 29.44]{BC17}. However, as we mentioned earlier in the introduction, the subgradient projection need not be even continuous; see \cite[Example 29.47]{BC17}. A comprehensive overview of cutters can be found in \cite[Chapter 2]{Ceg12}. Below we only recall \cite[Remark 2.1.31]{Ceg12}, which will be used in the sequel.

\begin{proposition}\label{prop:cutter}
  Let $T\colon \mathcal H \to \mathcal H$ be an operator with $\fix T \neq \emptyset$. Then $T$ is a cuttter if and only if $\langle T(x)-x, z-x\rangle \geq \|T(x)-x\|^2$ for all $x\in\mathcal H$ and $z\in \fix T$.
\end{proposition}

\section{Auxiliary Results} \label{sec:auxiliary}

\begin{lemma} \label{lem:BWWX15}
  Let $T\colon\mathcal H \to \mathcal H$ be a cutter, let $\alpha \in (0,2]$ and let $\rho\colon\mathcal H \to (0,\infty)$. Define the operator $U\colon\mathcal H \to \mathcal H$ by $U(x):=x+\alpha \beta(x) (T(x)- x)$, where
  \begin{equation}\label{lem:BWWX15:beta}
    \beta(x):=
    \begin{cases}\displaystyle
      \frac{\rho(x)+\|T(x)-x\|}{\|T(x)-x\|}, & \mbox{if } T(x) \neq x \\
      0, & \mbox{otherwise}.
    \end{cases}
  \end{equation}
  Assume that $x\notin \fix T$ and $B(y,\rho(x))\subseteq \fix T$. Then we have
  \begin{equation}\label{lem:BWWX15:ineq}
    \|U(x)-y\|^2 \leq \|x-y\|^2 - \frac{2-\alpha}{\alpha} \|U(x)-x\|^2.
  \end{equation}
\end{lemma}

\begin{proof}
  The argument follows the proof of \cite[Corollary 2.1(v)]{BWWX15}, which is only presented for a constant overrelaxation $\rho$.  Define
  \begin{equation}\label{pr:lem:BWWX15:z}
    w:= y-\rho(x)\frac{T(x)-x}{\|T(x)-x\|}
  \end{equation}
  and observe that $w\in B(y,\rho(x))\subseteq \fix T$. Since $T$ is a cutter, we have
  \begin{equation}\label{pr:lem:BWWX15:cutter}
    \rho(x)\|T(x)-x\| + \|T(x)-x\|^2  - \langle y-x,T(x)-x\rangle
    = \langle w- T(x), x- T(x)\rangle \leq 0.
  \end{equation}
  On the other hand, by \cite[Corollary 2.14]{BC17} applied to
  \begin{equation}\label{}
    u:= T(x)-y + \rho(x)\frac{T(x)-x}{\|T(x)-x\|} \text{\qquad and \qquad}
    v:=x-y,
  \end{equation}
  we obtain
  \begin{equation}\label{pr:lem:BWWX15:ineq1}
    \|U(x)-y\|^2=\|\alpha u + (1-\alpha)v\|^2 = \alpha \|u\|^2 + (1-\alpha)\|v\|^2 -\alpha(1-\alpha)\|u-v\|^2.
  \end{equation}
  By \eqref{pr:lem:BWWX15:cutter}, we have
  \begin{align}
    \nonumber
    \|u\|^2 & = \|T(x)-y\|^2 + \rho^2(x)
        + \frac{2\rho(x)}{\|T(x)-x\|} \langle (T(x)-x)+ (x-y), T(x)-x\rangle
    \\ \nonumber
    & = \|T(x)-y\|^2 + \rho^2(x) + 2\rho(x) \|T(x)-x\|
        - \frac{2\rho(x)}{\|T(x)-x\|}\langle y-x, T(x)-x\rangle
    \\ \nonumber
    & \leq \|T(x)-y\|^2 - \rho^2(x)
    \\ \nonumber
    & = \|(x-y)+(T(x)-x)\|^2 - \rho^2(x)
    \\ \nonumber
    & = \|x-y\|^2 -\|T(x)-x\|^2 + 2(\|T(x)-x\|^2 -\langle y-x, T(x)-x\rangle) - \rho^2(x)
    \\ \nonumber
    & \leq \|x-y\|^2 -\|T(x)-x\|^2 -2 \rho(x) \|T(x)-x\| - \rho^2(x)
    \\
    & = \|x-y\|^2- (\rho(x) + \|T(x)-x\|)^2.
  \end{align}
  Consequently,
  \begin{equation}
    \alpha \|u\|^2 \leq \alpha \|x-y\|^2- \frac{1}{\alpha}\|U(x)-x\|^2.
  \end{equation}
  Moreover,
    \begin{equation}\label{}
    \alpha(1-\alpha)\|u-v\|^2
    = \frac{1-\alpha}{\alpha}\|U(x)-x\|^2.
  \end{equation}
  Combining this with \eqref{pr:lem:BWWX15:ineq1}, we arrive at \eqref{lem:BWWX15:ineq}, which completes the proof.
\end{proof}

The following lemma is a key tool in our analysis.  We use it, in particular,  to derive estimate \eqref{int:descent}. Before proceeding, recall that $I_+(x):=\{i\in I \colon x \notin C_i\}$, $x\in\mathcal H$.

\begin{lemma} \label{lem:basic}
  Assume that $C_i=\fix T_i$ for given cutter operators $T_i\colon\mathcal H \to \mathcal H$, $i\in I$. Moreover, let $\alpha\in (0,2]$, let $\rho_i\colon\mathcal H \to (0,\infty)$, $i \in I$, and let $J\colon\mathcal H \to 2^I\setminus \{\emptyset\}$ satisfy $\sup_{x\in \mathcal H} \#(J(x))<\infty$. Furthermore, let $\lambda_j \colon \mathcal H \to [0,1]$ be such that $\sum_{j\in J(x)}\lambda_j(x)=1$. Define the operator $V\colon\mathcal H \to \mathcal H$ by
  \begin{equation}\label{lem:basic:V}
    V(x):=x+\alpha\sum_{j\in J(x)}\lambda_j(x)\beta_j(x)(T_j(x)- x),
  \end{equation}
  where
  \begin{equation}\label{}
    \beta_j(x):=
    \begin{cases}\displaystyle
      \frac{\rho_j(x)+\|T_j(x)-x\|}{\|T_j(x)-x\|}, & \mbox{if } T_j(x) \neq x \\
      0, & \mbox{otherwise}.
    \end{cases}
  \end{equation}

  Assume that $C\cap Q \neq \emptyset$ and that the weights $\lambda_j$ satisfy the inequality $\lambda_j(x)\geq \lambda>0$ for all $x \in \mathcal H$ and $j\in J_+(x) :=J(x)\cap I_+(x)$. Then
  \begin{equation}\label{lem:basic:FixV}
    \fix (P_Q V)=Q\cap \fix V
    \quad\text{and}\quad
    \fix V= \Big\{x \colon x\in \bigcap_{j\in J(x)}\fix T_j \Big\} .
  \end{equation}
  Moreover, assume that there are $z\in Q$ and $R>0$ such that $B(z,2R)\subseteq C$. Then for all $x \notin C$ with $\rho(x):=\max_{j\in J_+(x)}\rho_j(x)\leq R$, we have
  \begin{equation}\label{lem:basic:ineq}
    \|P_Q(V(x))-z\|^2 \leq \|x-z\|^2 - 2\alpha \lambda R \rho(x).
  \end{equation}
\end{lemma}

\begin{proof}
   We  first  show \eqref{lem:basic:FixV}. To this end, assume that $Q\cap C \neq\emptyset$ and define $F:=\{x \colon x\in \bigcap_{j\in J(x)}\fix T_j \}$. Observe that $C\subseteq F$ and thus $F\neq\emptyset$. Moreover, it is not difficult to see that the inclusion $Q\cap F \subseteq \fix P_Q V$ follows from the definition of $V$. It suffices to show that $\fix P_QV\subseteq Q\cap F$. Clearly, by the definition of the metric projection, $P_QV \subseteq Q$ and consequently, $\fix P_QV\subseteq Q$. Let $x \in \fix P_QV$ and suppose to the contrary that $x\notin F$, that is, $J_+(x)\neq\emptyset$. Since $P_Q$ is a cutter, we have
  \begin{equation}\label{pr:thm:basic:PQVx1}
    \langle Vx- P_Q(V(x)), z-  P_Q(V(x))\rangle \leq 0
  \end{equation}
  for all $z\in Q$. On the other hand, since each $T_j$ is a cutter, for all $x\in \mathcal H$ and $z\in \fix T_j$, we have, by Proposition \ref{prop:cutter},
  \begin{equation}\label{pr:thm:basic:cutterTi}
    \langle T_j(x)-x, z-x\rangle \geq \|T_j(x)-x\|^2.
  \end{equation}
  Since $x=P_Q(V(x))$, for any $z\in Q\cap C$, we arrive at
  \begin{align}\label{pr:thm:basic:PQVx2} \nonumber
    \langle V(x)-P_Q(V(x)),  z -P_Q(V(x))\rangle &=
    \langle V(x)-x, z-x\rangle  \\ \nonumber
    & = \alpha\sum_{j\in J_+(x)} \lambda_j(x)\beta_j(x) \langle T_j(x)-x, z-x \rangle \\
    & \geq \alpha\sum_{j\in J_+(x)}\lambda_j(x)\beta_j(x) \|T_j(x)-x\|^2  >0,
  \end{align}
  which is in contradiction with \eqref{pr:thm:basic:PQVx1}. Consequently $J_+(x)=\emptyset$ and $Q\cap F=\fix P_QV$, as claimed.

  Next we show that \eqref{lem:basic:ineq} holds for all $x\notin C$ with $\rho(x)\leq R$. To this end, for each $i\in I$, we define an auxiliary operator  $U_i$ by $U_i(x):=x+\alpha \beta_i(x)(T_i(x)-x)$.  Thus $V(x)=\sum_{j\in J(x)}\lambda_j(x) U_j (x)$. Let $x\notin C$ be such that $\rho(x)\leq R$ and let $j\in J_+(x)$. Observe that for any $y\in B(z,R)$, we have $B(y,\rho(x))\subseteq B(z,2R)\subseteq \fix T_j$. Consequently, by Lemma \ref{lem:BWWX15} applied to $U_j$, we have
  \begin{equation}\label{pr:thm:basic:ineq1}
    \|U_j(x)-y\|^2 \leq \|x-y\|^2 - \frac{2-\alpha}{\alpha}\|U_j(x)-x\|^2.
  \end{equation}
  In particular, the above inequality holds for
  \begin{equation}\label{pr:thm:basic:ineq2}
    y:=z-R\frac{U_j(x)-x}{\|U_j(x)-x\|},
  \end{equation}
  which by the choice of $j\in J_+(x)$ is well defined. By expanding the left-hand side of the inequality \eqref{pr:thm:basic:ineq1} with $y$ defined as above, we obtain
  \begin{align}\label{pr:thm:basic:ineq3} \nonumber
    \|U_j(x)-y\|^2 & = \left\|U_j(x)-z+R\frac{U_j(x)-x}{\|U_j(x)-x\|}\right\|^2\\
    &= \|U_j(x)-z\|^2 + \frac{2R}{\|U_j(x)-x\|}\langle U_j(x)-z, U_j(x)-x\rangle + R^2.
  \end{align}
  On the other hand,
  \begin{align}\label{pr:thm:basic:ineq4} \nonumber
    \|x-y\|^2 & = \left\|x-z+R\frac{U_j(x)-x}{\|U_j(x)-x\|}\right\|^2\\
    &= \|x-z\|^2 - \frac{2R}{\|U_j(x)-x\|}\langle z-x, U_j(x)-x\rangle + R^2
  \end{align}
  and
  \begin{equation}\label{pr:thm:basic:ineq5}
    \|U_j(x)-x\|=\alpha(\rho_j(x)+\|T_j(x)-x\|)\geq \alpha \rho_j(x).
  \end{equation}
  By combining \eqref{pr:thm:basic:ineq1} with \eqref{pr:thm:basic:ineq2}, \eqref{pr:thm:basic:ineq3} and \eqref{pr:thm:basic:ineq4}, we arrive at
  \begin{align}\label{pr:thm:basic:ineq6}
    \nonumber
    \|U_j(x)-z\|^2 &\leq \|x-z\|^2-2R\|U_j(x)-x\|-\frac{2-\alpha}{\alpha}\|U_j(x)-x\|^2
    \\ \nonumber
    &\leq \|x-z\|^2-\alpha \rho_j(x) \big(2R +(2-\alpha) \rho_j(x)\big)\\
    &\leq \|x-z\|^2-2 \alpha R \rho_j(x).
  \end{align}
  Observe that \eqref{pr:thm:basic:ineq6} holds for all $j\in J_+(x)$. Moreover, for all $j\in J(x)\setminus J_+(x)$, we have
  \begin{equation}\label{pr:thm:basic:ineq7}
    \|U_j(x)-z\| = \|x-z\|.
  \end{equation}
  Hence, by the nonexpansivity of the metric projection $P_Q$ and the convexity of the squared norm $\|\cdot\|^2$, we have
  \begin{align}\label{}
    \nonumber
    \|P_Q(V(x))-z\|^2
    &=\|P_Q(V(x))-P_Q(z)\|^2\leq \|V(x)-z\|^2 \\ \nonumber
    &\leq \sum_{j\in J(x)\setminus J_+(x)} \lambda_j(x) \|x-z\|^2 + \sum_{j\in J_+(x)} \lambda_j(x)\|U_j(x)-z\|^2
    \\ \nonumber
    & \leq \|x-z\|^2- 2\alpha R\sum_{j\in J_+(x)} \lambda_j(x) \rho_j(x)
    \\
    & \leq \|x-z\|^2-2\alpha\lambda R \rho(x).
  \end{align}
  The above inequality completes the proof.
\end{proof}

The following lemma  extends  \cite[Lemma 2.2]{BWWX15}.

\begin{lemma}[Slater Condition]\label{lem:subdiff}
  For each $i\in I$, let $f_i\colon\mathcal{H}\rightarrow\mathbb{R}$ be a convex and lower semicontinuous function, and assume that $f(z):=\sup_{i\in I} f_i (z)<0$ for some $z\in\mathcal H$. Then for all $r>0$, we have
  \begin{equation}\label{lem:subdiff:ineq}
    \inf \Big\{\|g_i(x)\|\colon  x\in B(z,r),\ f_i(x) > 0,\ g_i(x)\in\partial f_i(x)\Big\}
     \geq  \frac{-f(z)}{r}  =:\delta>0.
  \end{equation}
\end{lemma}

\begin{proof}
  Let $x\in B(z,r)$ and $i\in I_+(x)$. By the subgradient inequality, we have  $f_i(z)\geq f_i(x)+\langle g_i(x),z-x\rangle,$  where $g_i(x)\in\partial f_i(x)$. Since $i\in I_+(x)$, we have $f_i(x)> 0$ and consequently,
  \begin{equation}
    -f(z)
    \leq -f_i(z)
    \leq -f_i(x)+\langle g_i(x),x-z\rangle
    < \langle g_i(x),x-z\rangle
    \leq \| g_i(x)\|\|x-z\|
    \leq \|g_i(x)\|r,
  \end{equation}
  from which \eqref{lem:subdiff:ineq} follows.
\end{proof}

\section{Deterministic Methods}\label{sec:mainResultsDet}
In this section we return to the CFP defined in the introduction.

\begin{definition}\label{def:control:det}
  We call the sequence $\{I_k\}_{k=0}^\infty$ a \emph{control sequence in $I$} if each $I_k\colon\mathcal H \to 2^I\setminus \{\emptyset\}$ is a set-valued mapping with  $M:=\sup_{x,k}\#(I_k(x))<\infty$. If each $I_k$ is single-valued, say $I_k(x)=\{i_k(x)\}$, where $i_k\colon\mathcal H \to I$, then we also call the sequence $\{i_k\}_{k=0}^\infty$ a \emph{control sequence} in $I$. We say that the control sequence in $I$ is \textit{nonadaptive} if each set-valued mapping $I_k$ (single-valued mapping $i_k$) is constant, that is, when $I_k(x)=I_k(y)$ ($i_k(x)=i_k(y)$) for all $x,y\in\mathcal H$. In this case we omit the argument.
\end{definition}

\begin{definition}\label{def:matchedControl}
  We say that the control sequence $\{I_k\}_{k=0}^\infty$ in $I$ is \emph{well matched} with the set $C$ if
  \begin{equation}\label{def:matchedControl:multi}
    \#(x, \{I_k\}_{k=0}^\infty)
    :=\#(\{k \geq 0 \colon  I_k(x)\cap I_+(x)\neq\emptyset \})=\infty
  \end{equation}
  for all $x\notin C$. In particular, a single-valued control sequence $\{i_k\}_{k=0}^\infty$ in $I$ is \emph{well matched} with $C$ if
  \begin{equation}\label{def:matchedControl:single}
    \#(x, \{i_k\}_{k=0}^\infty)
    :=\#(\{k \geq 0 \colon  i_k(x) \in I_+(x)\})=\infty
  \end{equation}
  for all $x\notin C$.
\end{definition}

\begin{example}\label{ex:controls_MV}
  In the literature (see \cite{BB96, Ceg12}) one can find the following examples  of a nonadaptive control sequences $\{I_k\}_{k=0}^\infty$ in $I$ which are well matched with $C$: (a) \emph{cyclic control} defined in a finite $I$ which satisfies  $I_k = I_{(k\Mod s)}$ for all $k=0,1,2,\ldots$, where $s\geq 2$;  (b) \emph{intermittent control} defined in a finite $I$ which satisfies $I=\bigcup_{k=n}^{n+s-1}I_k$ for all $n = 0,1,2,\ldots$ and some integer $s\geq 2$;  and the more general (c) \emph{repetitive control} defined in both finite and infinite $I$, which satisfies $I=\bigcup_{k=n}^{\infty}I_k$ for all $n = 0,1,2,\ldots$.  All three definitions simplify when reduced to a single-valued control sequences $\{i_k\}_{k=0}^\infty$.
\end{example}

\begin{proposition}\label{thm:cardCond}
  Let $\{I_k\}_{k=0}^\infty$ be a nonadaptive control sequence in $I$ and consider the following statements:
  \begin{enumerate}[(i)]
    \item $\{I_k\}_{k=0}^\infty$ is well matched with $C$.
    \item $F_n:=\bigcap_{k=n}^\infty \bigcap_{i\in I_k} C_i =C$ for all $n=0,1,2,\ldots$.
    \item $\{I_k\}_{k=0}^\infty$ is repetitive in $I'$ for some $\emptyset\neq I'\subseteq I$ (that is, $I'\subseteq \bigcup_{k=n}^\infty I_k$ for all $n=0,1,2,\ldots$) and $C=\bigcap_{i\in I'}C_i$ .
  \end{enumerate}
  Then (i)$\Leftrightarrow$(ii)$\Leftarrow$ (iii). Moreover, if $I$ is finite, then (ii)$\Rightarrow$ (iii).
\end{proposition}

\begin{proof}
  We first show that (i) implies (ii). Suppose to the contrary that the control sequence $\{I_k\}_{k=0}^\infty$ is well matched with $C$ and that for some $n\geq 0$, there exists a point $x \in F_n \setminus C$. Then $x\in C_i$ for all $i\in I_k$ and $k\geq n$, that is, $\#(x, \{I_k\}_{k=0}^\infty)\leq n$. On the other hand, since $x\notin C$, by condition \eqref{def:matchedControl:multi}, $\#(x, \{I_k\}_{k=0}^\infty)=\infty$, which is a contradiction.

  To show that (ii) implies (i), assume that the equality $F_n=C$ holds for all $n=0,1,2,\ldots$ and let $x\notin C$. Then for each $n=0,1,2,\ldots$, consider the smallest $k_n \geq n$ such that $x\notin F_{k_n}$. By the definition of the set $F_n$ and by eventually passing to a subsequence, we may assume that $x\notin \bigcap_{i\in I_{k_n}}C_i$ in view of which, $\#(x,\{I_k\}_{k=0}^\infty)=\infty$.

  It is not difficult to see that (iii) implies (i). Indeed, if $x\notin C$, then, by (iii), $x$ violates at least one constraint $C_i$ for some $i\in I'$. Since the control is repetitive in $I'$, we see that $i\in I_k$ for infinitely many $k$'s and thus $\#(x, \{I_k\}_{k=0}^\infty)=\infty$.

  Assume now that $I$ is finite. We show that (iii) follows from (ii). To this end, define  $I':=\limsup_{k\to \infty}I_k=\bigcap_{n=0}^\infty\bigcup_{k=n}^\infty I_k$  and observe that $i\in I'$ if and only if $i\in I_k$ for infinitely many $k$'s. Since $I$ is finite, we see that $I'\neq \emptyset$. Assume that $I'$ is a proper subset of $I$. For each $i\in I\setminus I'$, there is $n_i\geq 0$ such that $i\notin \bigcup_{k=n_i}^\infty I_k$ and since $I$ is finite, we have $n:=\max_{i\in I\setminus I'}n_i<\infty$. Consequently, $I'=\bigcup_{k=n}^\infty I_k$ and, by (ii), we arrive at $C=F_n=\bigcap_{i\in I'}C_i$. This completes the proof.
\end{proof}

\begin{remark}
  The implication (i)$\Rightarrow$ (iii) may not be true when $I$ is infinite. To see this, consider a decreasing sequence of sets $C_{k+1}\subseteq C_k$ with nonempty intersection $C$ and define $i_k:=k$. The control $\{i_k\}_{k=0}^\infty$ is well matched with $C$, but clearly it is not repetitive in any subset $I'\subseteq I$.
\end{remark}

\begin{example}
  Observe that in view of Definition \ref{def:matchedControl}, if for all $x\in \mathcal H$ and all $k=0,1,2,\ldots$, the set $I_k(x)$ contains at least one index from the set of violated constraints $I_+(x)$, then the control sequence $\{I_k\}_{k=0}^\infty$ is well matched with $C$.  In particular, when $I$ is finite, one could use maximal control sequences such as: (a) the \emph{remotest set control} $i_k(x) := \argmax_{i\in I}d(x,C_i)$; (b) the \emph{maximal displacement control} $i_k(x) :=\argmax_{i\in I}\|T_i(x)-x\|$, when $C_i = \fix T_i$; or (c) the \emph{maximal violation control} $i_k(x) := \argmax_{i\in I} f_i^+(x)$, when $C_i = S(f_i,0)$.
\end{example}

\begin{theorem}\label{thm:main}
  Assume that $C_i=\fix T_i$ for given cutter operators $T_i\colon\mathcal H \to \mathcal H$, $i\in I$. Let $\{I_k\}_{k=0}^\infty$ be a given control sequence in $I$ and let the weights $\lambda_{i,k} \colon \mathcal H \to  [0,1]$ be such that $\sum_{i\in I_k(x)}\lambda_{i,k}(x)=1$. Moreover, let $\{\alpha_k\}_{k=0}^\infty\subset (0,2]$ be a sequence of relaxations and let $\{r_k\}_{k=0}^\infty\subset (0,\infty)$ be a sequence of overrelaxations.  Finally, let $\varphi_i\colon\mathcal H \to (0,\infty)$.  Define the sequence $\{x_k\}_{k=0}^\infty$ by
  \begin{equation}\label{thm:main:xk}
    x_0\in Q,\quad x_{k+1}:=
    P_Q\left(x_k+\alpha_{[k]}\sum_{i\in I_k(x_k)}
    \lambda_{i,k}(x_k)
    \beta_{i,k}(x_k)
    \Big(T_i(x_k)-x_k\Big)\right),
  \end{equation}
  where
  \begin{equation}\label{thm:main:k}
    [0]:=0, \quad [k]:=\#(\{0\leq n \leq k-1\colon x_{n}\neq x_{n+1}\}), \quad k=1,2,\ldots
  \end{equation}
   and
  \begin{equation}\label{thm:main:betak}
    \beta_{i,k}(x):=
    \begin{cases}\displaystyle
      \frac{\frac{r_{[k]}}{\varphi_i(x)} +\|T_i (x)-x\|}{\|T_i(x)-x\|}, & \mbox{if } T_i(x)\neq x \\
      0, & \mbox{otherwise}
    \end{cases}, \quad x\in\mathcal H.
  \end{equation}
  Assume that
  \begin{enumerate}[(i)]
    \item $\interior(C)\cap Q \neq \emptyset$.
    \item $r_k \to 0$ and $\sum_{k=0}^\infty\alpha_kr_k=\infty$.
    \item For each bounded subset $S \subset \mathcal H$, there are $\delta, \Delta \in (0,\infty)$ s.t. $\delta \leq \varphi_i(x) \leq \Delta$ for all $x \in S$.
    \item There is $\lambda > 0$ such that $\lambda_{i,k}(x)\geq \lambda>0$ for all $x$, $k$ and $i\in I_k(x)\cap I_+(x)$.
    \item $\{I_k\}_{k=0}^\infty$ is well matched with the set $C$.
  \end{enumerate}
  If  the sequence $\{x_k\}_{k=0}^\infty$ is bounded (see Examples \ref{ex:phi1} and \ref{ex:phi2}),  then $x_k\in C\cap Q$ for some $k$.
\end{theorem}

\begin{proof}
  Since $x_k \in Q$, it suffices to show that $x_k \in C$ for some $k$. For each $k=0,1,2,\ldots$, define the operator $V_k\colon\mathcal H \to \mathcal H$ by
  \begin{equation}\label{pr:main:Vk}
    V_k (x):= x+ \alpha_{[k]}\sum_{i\in I_k^+(x)}
    \lambda_{i,k}(x)
     \beta_{i,k}(x)  \Big(T_i (x)- x\Big),
  \end{equation}
  where $I_k^+(x):=I_k(x)\cap I_+(x)$.  Clearly, we can write $x_{k+1}=P_Q (V_k (x_k)) $. We divide the rest of the proof into two cases.

  \textit{Case 1.} Assume that $n:=\sup_{k\geq 0} [k]<\infty$, in which case $x_n=x_k$ for all $k\geq n$. We show that $x_n\in C$. By Lemma \ref{lem:basic}, the equality $x_{k+1}=x_k$ implies that $x_k \in \fix V_k$, that is, $x_k\in \bigcap_{i\in I_k(x_k)}C_i$. Consequently $\#(x_n,\{I_k\}_{k=0}^\infty) \leq n$ and thus $x_n \in C$. Otherwise, since the control $\{I_k\}_{k=0}^\infty$ is well matched with $C$, we would get $\#(x_n,\{I_k\}_{k=0}^\infty) = \infty$, a contradiction.

  \textit{Case 2.} Assume now that $\sup_{k\geq 0}  [k] =\infty$, that is, the set $ \mathcal N:=\{ n \geq 0 \colon  x_{n}\neq x_{n+1}\}$ is infinite.  There is a point $z\in \interior(C)\cap Q$ and a radius $R>0$ such that the ball $B(z,2R)\subseteq C$. Since the sequence $\{x_k\}_{k=0}^\infty$ is assumed to be bounded, there are $0<\delta \leq \Delta <\infty$ such that $\delta \leq \varphi_i(x_k) \leq \Delta$ for all $k=0,1,2,\ldots$. Consequently, the fraction $r_{[k]}/\varphi_i(x_k)$ can be made arbitrarily small because of the estimate $r_{[k]}/\varphi_i(x_k) \leq r_{[k]}/\delta$. In particular, we may assume that $r_{[k]}/\varphi_i(x_k) \leq R$ for all large enough $k$, say $k \geq K$ and $K \in \mathcal N$.
  Thus, by Lemma \ref{lem:basic} applied to $V:=V_k$ and  $\rho_i(x):=r_{[k]}/\varphi_i(x)$,  we get
  \begin{equation}\label{pr:main:ineq1}
    \|x_{k+1}-z\|^2
    \leq \|x_k-z\|^2 - 2\lambda R \alpha_{[k]}  \max_{i\in I_k^+(x_k)}\frac{r_{[k]}}{\varphi_i(x_k)}
    \leq \|x_k-z\|^2 - 2\lambda R \alpha_{[k]}  \frac{r_{[k]}}{\Delta}
  \end{equation}
  for all $k\in \mathcal N$, $k \geq K$ (compare with \eqref{int:descent}).  On the other hand,
  \begin{equation}\label{pr:main:ineq2}
    \|x_{k+1}-z\|^2 = \|x_k-z\|^2
  \end{equation}
  for all $k\notin \mathcal N$. Consequently, by inductively applying \eqref{pr:main:ineq1} and \eqref{pr:main:ineq2}  to large enough $k$ ($k \geq K$),  we obtain
  \begin{equation}\label{pr:main:ineq3}
    \|x_{k+1}-z\|^2
    \leq \|x_K -z\|^2 - \frac{2\lambda R}{\Delta} \sum_{\substack{n=K\\n\in \mathcal N}}^k \alpha_{[n]} r_{[n]}
  \end{equation}
  and therefore,  since $K\in \mathcal N$, we arrive at
  \begin{equation}\label{pr:main:contradiction}
    \frac{2\lambda R}{\Delta }
    \sum_{n=[K]}^{[k+1]-1} \alpha_n r_n
    = \frac{2\lambda R}{\Delta } \sum_{\substack{n=K\\n\in \mathcal N}}^k \alpha_{[n]} r_{[n]}
    \leq \|x_{K}-z\|^2 .
  \end{equation}
  Since, by assumption, $\sup_{k\geq 0}  [k] =\infty$, we see that the left-hand side in \eqref{pr:main:contradiction} tends to infinity as $k\to \infty$, which is a contradiction. Consequently, we must have $n=\sup_{k\geq 0}  [k]<\infty$, in which case we have already shown that $x_n\in C\cap Q$.
\end{proof}

There are at least two known situations, where we can ensure that the sequence $\{x_k\}_{k=0}^\infty$ is indeed bounded.

\begin{example}  \label{ex:phi1}
  In the setting of Theorem \ref{thm:main}, for each $x \in \mathcal H$, define
  \begin{equation}\label{}
    \varphi_i(x) := 1.
  \end{equation}
  Obviously, $\varphi_i(x)$ satisfies (iii). Furthermore, by applying Lemma \ref{lem:basic}, it is not difficult to see that the sequence $\{x_k\}_{k=0}^\infty$ is bounded. Consequently, $x_k\in C\cap Q$ for some $k$.
\end{example}

\begin{example} \label{ex:phi2}
  In the setting of Theorem \ref{thm:main}, assume that $C_i=\{x\colon  f_i(x)\leq 0\}$ for some convex and lower semicontinuous functions $f_i\colon\mathcal H \to \mathbb R$, $i\in I$. Let $T_i := P_{f_i}$ and let $g_i\colon \mathcal H \to \mathcal H$ be the associated subgradient mapping (see Example \ref{ex:subProj}). For each $x\in \mathcal H$, define
  \begin{equation}\label{ex:phi2:eq}
    \varphi_i(x):=
    \begin{cases}
      \|g_i(x)\|, & \mbox{if } f_i(x)>0 \\
      1, & \mbox{otherwise}.
    \end{cases}
  \end{equation}
  Using \eqref{ex:subProj:eq} and the convention that the summation over the empty set is zero, method \eqref{thm:main:xk} becomes
  \begin{equation}\label{ex:phi2:xk}
    x_0\in Q, \quad x_{k+1} =
    P_Q \left(x_k-\alpha_{[k]} \sum_{i\in I_k^+(x_k)}\lambda_{i,k}(x_k)\frac{r_{[k]} + f_{i}(x_k)}{\|g_i(x_k)\|^2}g_i(x_k)\right).
  \end{equation}
  Assume that (i') $f(z):= \sup_{i\in I} f_i(z) <0$ for some $z\in Q$; (ii') = (ii); and (iii') $\bigcup_{i\in I} \partial f_i (S)$ is bounded for bounded subsets $S\subset \mathcal H$. Then the sequence $\{x_k\}_{k=0}^\infty$ is bounded and assumptions (i')--(iii') imply conditions (i)--(iii). Consequently, $x_k\in C\cap Q$ for some $k$.
\end{example}

\begin{proof}
  We first demonstrate that the sequence $\{x_k\}_{k=0}^\infty$ is bounded. Obviously, the statement is trivial when $n:=\sup_{k\geq 0} [k]<\infty$. Assume now that $\sup_{k\geq 0} [k] =\infty$. It suffices to show that $\|x_{k+1}-z\|\leq\|x_k-z\|$ for all $k$ large enough.

  If $I_k^+(x_k)=\emptyset$, then $x_{k+1}=x_k$ and thus $\|x_{k+1}-z\|=\|x_k-z\|$. On the other hand, if $I_k^+(x_k)\neq\emptyset$, then, by using the nonexpansivity of the metric projection $P_Q$  and the convexity of $\|\cdot\|^2$, we get
  \begin{align} \nonumber
    \|x_{k+1}-z\|^2
    &\leq\left\|(x_k-z) - \alpha_{[k]} \sum_{i\in I_k^+(x_k)}\lambda_{i,k}(x_k)\frac{f_i(x_k) + r_{[k]}}{\|g_i(x_k)\|^2}g_i(x_k)\right\|^2\\
    &\leq \|x_k-z\|^2 + \alpha_{[k]} \sum_{i\in I_k^+(x_k)}\lambda_{i,k}(x_k) \theta_{k,i}(x_k)
    \frac{f_i(x_k) + r_{[k]}}{\|g_i(x_k)\|^2},
  \end{align}
  where
  \begin{equation}\label{}
    \theta_{k,i}(x_k) := \alpha_{[k]}(f_i(x_k)+ r_{[k]})
      -2\left\langle x_k-z,g_i(x_k)\right\rangle, \quad i\in I_k^+(x_k).
  \end{equation}
  Moreover, by combining the subgradient inequality with the inequality $ r_{[k]}\leq -f(z)$, which holds for all $k$ large enough (since $ r_{[k]}\to 0$ as $k\to \infty$), we get
  \begin{equation}\label{}
    \langle x_k-z,g_i(x_k)\rangle
    \geq f_i(x_k)-f_i(z)\geq f_i(x_k)-f(z)
    \geq f_i(x_k)+r_{[k]}.
  \end{equation}
  Consequently, $\theta_{i,k}(x_k) \leq -(2-\alpha_{[k]})(f_i(x_k)+r_{[k]})\leq 0$, which leads to $\|x_{k+1}-x\| \leq \|x_k-z\|$, as asserted.

  Let $S$ be nonempty and bounded subset of $\mathcal H$ and choose $r>0$ so that $S\subset B(z,r)$.  The assumed boundedness of the subdifferential and Lemma \ref{lem:subdiff} imply that there are $\Delta\geq \delta>0$ such that $\delta \leq \|g_i (x)\| \leq \Delta$, where the first inequality holds for all  $x\in S$ and all $i\in I_+(x)$,  whereas the second one holds for all  $x\in S$  and all $i\in I$.  This, when combined with \eqref{ex:phi2:eq}, implies condition (iii).
\end{proof}

\begin{remark}[Constant overrelaxations]
  A careful analysis of the proof of Theorem \ref{thm:main} shows that when $\varphi_i(x) = 1$, $i\in I$, we can use constant overrelaxations $r_k := r \leq R$. This, however, requires the explicit knowledge of the radius $R$, which is not always possible.
\end{remark}

\begin{remark}[{Using $k$ instead of $[k]$}]\label{rem:squareBrackets}
 Observe that if $I_k^+(x) \neq \emptyset$ for all $x\notin C$, then $[k]=k$ as long as $x_k\notin C$. This condition is trivially satisfied if, for example, the number of constraints $m=1$ or when $I_k = I$ for all $k=0,1,2,\ldots$. Furthermore, the square brackets can be partially omitted if $\alpha_k \geq \alpha >0$ and $\sum_{k=0}^{\infty}r_k =\infty$. In this case we may replace ``$\alpha_{[k]}$'' by ``$\alpha_k$'' in \eqref{thm:main:xk} without losing the finite convergence property. If, in addition, we assume that the sequence $\{r_k\}_{k=0}^\infty$ is decreasing and the control is $s$-intermittent for some $s\geq 1$, then we may fully drop the square brackets notation and write ``$r_k$'' instead of ``$r_{[k]}$'' in \eqref{thm:main:betak}. This, however, requires a more detailed discussion, which we present below. We note that if $r_k$ is not monotone, then by using $r_k$ instead of $r_{[k]}$, we may indeed lose the finite convergence property; see Example \ref{ex:NoFiniteConv}.
\end{remark}

\begin{proof}
  For each $k=0,1,2,\ldots,$ define $V_k$ in the same way as in \eqref{pr:main:Vk} with ``$[k]$'' replaced by ``$k$'' and assume that $\mathcal N$ is infinite. Since the control is $s$-intermittent, we have $\{k,k+1,\ldots,k+s-1\}\cap \mathcal N \neq \emptyset$ for all $k=0,1,2,\ldots$. Since  $r_{k+l} \geq r_{k+s}$ for all $l\in\{0,\ldots,s-1\}$,  inequality \eqref{pr:main:contradiction} becomes
  \begin{equation}\label{}
    \frac{2\alpha\lambda R}{{\Delta }}
     \sum_{n=1}^{\lfloor (k-K)/s \rfloor +1} r_{K+ns}
    \leq \frac{2\lambda R}{{\Delta }}
    \sum_{\substack{n=K\\n\in \mathcal N}}^k \alpha_{n} r_{n}
    \leq \|x_0-z\|^2
  \end{equation}
  for all $k\geq K$.  Observe that  due to monotonicity of the $r_k$'s,  the left-hand side tends to infinity as $k \to \infty$,  which  leads to a contradiction. Thus $\mathcal N$ has to be finite and by using the same argument as in the proof of Theorem \ref{thm:main}, we see that $x_k\in C\cap Q$ for some $k$.
\end{proof}

\begin{remark}[]
  The uniform boundedness of the subdifferential on bounded sets, which corresponds to condition (iii') of Example \ref{ex:phi2}, is a rather standard assumption; see \cite[Proposition 7.8]{BB96}. Note that condition (iii') is not mentioned explicitly in \cite{DePI88, IM86}, but only in \cite[Remark 16]{CCP11}. Nevertheless, it is satisfied therein because the set $I$ is finite and $\mathcal H=\mathbb R^n$. Thus Example \ref{ex:phi2} (in view of Remark \ref{rem:squareBrackets}) improves upon the results established in \cite{CCP11, DePI88, IM86}.
\end{remark}

\begin{remark}[{Comparison with \cite{BWWX15}}] \label{rem:comparisonWithBWWX15}
  Assume that $C = \fix T$, where $T\colon\mathcal H \to \mathcal H$ is a cutter ($m = 1$) and $\varphi(x):=1$. Then \eqref{thm:main:xk} becomes the iterative method proposed in \cite[equation (3)]{BWWX15}. Furthermore, Theorem \ref{thm:main} guarantees finite convergence if one only assumes that $r_k \to 0$, $\sum_{k=0}^\infty\alpha_kr_k=\infty$ and $C\cap \interior(Q) \neq \emptyset$. This extends both \cite[Theorems 3.1 and Theorem 3.2]{BWWX15}; compare with Table \ref{table}.
\end{remark}

\section{Stochastic Methods}\label{sec:mainResultsStoch}
In this section we consider a stochastic version of Theorem \ref{thm:main}. Let $(\Omega, \mathcal F, \prob)$ be a given probability space.

\begin{definition}\label{def:control:stoch}
  We call the sequence $\{I_k\}_{k=0}^\infty$ a \emph{random control sequence in} $I$ if $I_k\colon \Omega \to 2^I\setminus \{\emptyset\}$ are independent and identically distributed (set-valued) random variables on $(\Omega, \mathcal F, \prob)$ with $M:=\sup_{\omega,k}\#(I_k(\omega))<\infty$. If each $I_k(\omega)$ is single-valued, say $I_k(\omega) = \{i_k(\omega)\}$ for $i_k\colon \Omega \to I$, then we also call the sequence $\{i_k\}_{k=0}^\infty$ a \emph{random control} in $I$.
\end{definition}

\begin{remark}
  The phrase ``identically distributed'' means that
  $\prob(\{\omega\in \Omega \colon  I_k(\omega)=J\})=\prob(\{\omega\in \Omega \colon  I_n(\omega)=J\})$
  for all $k,n$ and all nonempty $J\subseteq I$ with $\#(J)\leq M$. The phrase ``independent'' means that
   $\prob ( \bigcap_{k\in K} \{\omega \in \Omega \colon  I_k(\omega)=J_k\} )
    = \prod_{k\in K} \prob (\{\omega \in \Omega \colon  I_k(\omega)=J_k\})$
  for all finite $K$ and all nonempty $J_k\subseteq I$ with $\#(J_k)\leq M$.
\end{remark}

Before formulating our next result, we establish a very intuitive lemma in view of which a random control is repetitive almost surely. We recall that $\{I_k(\omega)\}_{k=0}^\infty$ is repetitive in $I'$ if $I'\subseteq \bigcup_{k=n}^\infty I_k(\omega)$ for all $n=0,1,2,\ldots$.

\begin{lemma}\label{lem:repetitive}
  Let $\{I_k\}_{k=0}^\infty$ be a random control in $I$ and assume that $\prob(\{\omega\in \Omega \colon  i\in I_k(\omega)\})>0$ for all $i\in I'\subseteq I$. Then $\prob \left( \{\omega\in \Omega \colon  \{I_k(\omega)\}_{k=0}^\infty \text{ is repetitive in } I'\}\right)= 1.$
\end{lemma}

\begin{proof}
  Define the events $A_i^k:=\{\omega \in \Omega \colon  i\in I_k(\omega)\} $ and the family $    \mathcal J:=\{J\subseteq I \colon i\in J\text{ and } \# (J)\leq M \}$, where $M:=\sup_{\omega,k}\#(I_k(\omega))$. Since the events $A_J^k:=\{\omega \in \Omega \colon  I_k(\omega)=J\}$ are disjoint for different values of $J \in \mathcal J$, we have
  \begin{equation}\label{}
    \prob (A_i^k)
    = \prob\Big( \bigcup_{J\in \mathcal J} A_J^k\Big)
    = \sum_{J\in \mathcal J} \prob(A_J^k).
  \end{equation}
  By assumption, the variables $I_k$ are identically distributed.  Consequently,  we see that $\prob (A_i^k)=\prob(A_i^n)$ for all $i\in I$ and all $k,n=1,2,\ldots$. Hence the probability $p_i:=\prob(A_i^k)$ does not depend on $k$ and, by assumption, $p_i>0$ for all $i\in I'$.

  Moreover, the events $A_i^k$ are independent over $k$. For simplicity, we only show this for a pair $K=\{k,n\}$, $k\neq n$, although the argument holds for any finite set of indices $K$. Indeed, by disjointness (over $J$) and independence (over $k$) of the events $A_J^k$, we have
  \begin{align}\label{}\nonumber
    \prob(A_i^k\cap A_i^n)
    & = \prob \Big( \bigcup_{J\in\mathcal J} A_J^k
        \cap \bigcup_{J'\in\mathcal J} A_{J'}^n)\Big)
    = \prob \Big( \bigcup_{J,J'\in\mathcal J} (A_J^k \cap A_{J'}^n)\Big)
    \\&=\sum_{J,J'\in\mathcal J} \prob (A_J^k) \prob(A_{J'}^n)
    =\prob (A_i^k) \prob(A_i^n).
  \end{align}
  Consequently, for all $i\in I'$, we obtain $\sum_{k=0}^{\infty} \prob(A_i^k)=\sum_{k=0}^{\infty} p_i=\infty$ and, by applying the Borel-Cantelli lemma (see \cite[Theorem 8.3.4]{Dud04}) to $A_i := \limsup_{k\to\infty} A_i^k$, we have $\prob(A_i)=1$.

  Consider a decreasing sequence of sets $ E_{k} := \bigcap_{t=1}^k A_{i_t} $ where $k=1,2,\ldots,n$ for finite $I'=\{i_1,\ldots,i_n\}$, whereas $k=1,2,\ldots$ for infinite $I'=\{i_1,i_2,\ldots\}$. Clearly, the set $A_i$ consists of all $\omega \in \Omega$ for which the membership $i\in I_k(\omega)$ happens infinitely many times in the sequence $\{I_k(\omega)\}_{k=0}^\infty$. Bearing this in mind, we get
  \begin{equation}
    E := \bigcap_{k=1}^{\#(I')} E_k
    =\bigcap_{i\in I'} A_i
    = \{\omega\in \Omega \colon  \{I_k(\omega)\}_{k=0}^\infty \text{ is repetitive in $I'$}\}.
  \end{equation}

  We now show, by induction, that $\prob (E_k)=1$ for all $k$. Indeed, by definition, $\prob(E_1)=\prob(A_{i_1})=1$, as we have already observed above. Moreover,
  \begin{equation}\label{}
    \prob (E_{k+1}) = \prob(E_k \cap A_{i_{k+1}})
    = \prob(E_k) + \prob(A_{i_{k+1}}) - \prob(E_k \cup A_{i_{k+1}}).
  \end{equation}
  By induction, $\prob(E_k)=1$ and, since $\prob(A_{i_{k+1}})=1$, we conclude that $\prob(E_k \cup A_{i_{k+1}})=1$. Otherwise $\prob(E_{k+1})$ would be greater than one. Hence $ \prob(E_{k+1})  =1$, as asserted.

  Observe that if $I'$ is finite, then $E=E_n$ and consequently, $\prob(E) = \prob(E_n)=1$. On the other hand, if $I'$ is infinite, then $\{E_k\}_{k=0}^\infty$ is a decreasing sequence of events, where $ E_{k+1}\subseteq E_k$ and, by the continuity of $\prob$ (see \cite[Chapter II.9, Theorem E]{Hal50}), we have
  $
    \prob (E) = \lim_{k\to \infty}\prob (E_k)=1.
  $
  This completes the proof.
\end{proof}

\begin{theorem}\label{thm:main:stoch}
  Let the sequence $\{x_k\}_{k=0}^\infty$ be defined as in Theorem \ref{thm:main} using a random control $\{I_k\}_{k=0}^\infty$ in $I$ (that is, for each $\omega \in \Omega$ we define a sequence $\{x_k(\omega)\}_{k=0}^\infty$, where at each step $k$ the iterate $x_{k+1}(\omega)$ is obtained by using the realization $I_k(\omega)$ instead of $I_k(x_k)$, and where $x_0(\omega)=x_0\in Q$). Assume that conditions (i)--(iv) of Theorem \ref{thm:main} hold and that for all $x\notin C$, we have
  \begin{equation}\label{thm:main:stoch:PrAssumption}
    \prob(\{\omega\in \Omega \colon  I_k(\omega)\cap I_+(x)\neq\emptyset\})>0.
  \end{equation}
  Then, by \eqref{thm:main:stoch:PrAssumption}, almost surely, the control sequence $\{I_k\}_{k=0}^\infty$ is well matched with $C$ (condition (v)). Furthermore, almost surely, $x_k \in C\cap Q$ for some $k$ given that $\{x_k\}_{k=0}^\infty$ is bounded (with some positive probability).
\end{theorem}

\begin{proof}
  Define the following events:
  $E_0 :=  \{\omega \in \Omega \colon  \{x_k(\omega)\}_{k=0}^\infty
     \text{ is bounded}\}$, $E_1 := \{\omega \in \Omega \colon  \{I_k(\omega)\}_{k=0}^\infty$ $
     \text{ is repetitive in } I'\}$, $E_2 := \{\omega \in \Omega \colon \{I_k(\omega)\}_{k=0}^\infty$ is well matched with $C\}$ and $E_3 := \{\omega \in \Omega \colon  x_k(\omega) \in C\cap Q \text{ for some }k \}$, where, as in Lemma \ref{lem:repetitive}, $I':=\{i\in I \colon \prob(A_i^k)>0\}$ and $A_i^k:=\{\omega \in \Omega \colon  i\in I_k(\omega)\}$.
  It suffices to show that $\prob(E_3 \mid E_0) = 1$.

  We first demonstrate that  $C=\bigcap_{i\in I'}C_i$. Suppose to the contrary that $x\in \bigcap_{i\in I'}C_i \setminus C$. Clearly, $I_+(x)\subseteq I\setminus I'$ and thus $\prob (A_i^k)=0$ for all $i\in I_+(x)$. By \eqref{thm:main:stoch:PrAssumption}, we obtain
  \begin{equation}
    0 < \prob(\{\omega\in \Omega \colon I_k(\omega)\cap I_+(x)\neq \emptyset\})
    =  \prob\Big( \bigcup_{i\in I_+(x)}A_i^k\Big)
    \leq \sum_{i\in I_+(x)}\prob(A_i^k)=0,
  \end{equation}
  a contradiction.

  Consequently, by Proposition \ref{thm:cardCond}, we get $E_1 \subseteq E_2$.  Moreover, by Lemma \ref{lem:repetitive}, we get $\prob(E_1) = 1$. Thus $\prob(E_2) = 1$, but also $\prob(E_0 \cup E_2) = 1$. On the other hand, by Theorem \ref{thm:main}, we have $E_0 \cap E_2 \subseteq E_3 = E_0 \cap E_3$, where, by assumption, $\prob(E_0)>0$. Hence,
  \begin{equation}\label{}
    \prob(E_0\cap E_3) \geq \prob(E_0 \cap E_2)
    = \prob(E_0) + \prob(E_2) - \prob(E_0 \cup E_2)
    = \prob(E_0)
  \end{equation}
  and we arrive at $\prob(E_3 \mid E_0) = \prob(E_0\cap E_3)/\prob(E_0) \geq \prob(E_0)/\prob(E_0) =1$.
\end{proof}

\begin{remark}[{Comparison with \cite{Pol01}}]\label{ex:Methods_with_StochSVC}
  If $\{i_k\}_{k=0}^\infty$ is a single-valued random control in $I$. Then condition \eqref{thm:main:stoch:PrAssumption} recovers \cite[Assumption 2]{Pol01}, that is, $\prob(\{\omega\in \Omega \colon  i_k(\omega)\in I_+(x)\})>0$ for all $x\notin C$. In particular, Theorem \ref{thm:main:stoch} recovers \cite[Theorem 1]{Pol01} reduced to at most countably infinite number of constraints.
\end{remark}

\textbf{Acknowledgements.} We thank two anonymous referees for their useful comments and helpful suggestions. This work was partially supported by the Israel Science Foundation (Grants 389/12 and 820/17), the Fund for the Promotion of Research at the Technion and by the Technion General Research Fund.

\vspace{-1em}

\bibliographystyle{siam}
\footnotesize
\addcontentsline{toc}{section}{References}
\bibliography{references}

\normalsize
\vspace{-1em}
\section*{Appendix} \label{sec:Appendix}
\addcontentsline{toc}{section}{Appendix}
\begin{example}\label{ex:NoFiniteConv}
  We show that the  relaxed alternating projection method, where $\alpha = 1/2$,  may fail to converge in finitely many steps if the sequence of overrelaxations $\{r_k\}_{k=0}^\infty$ is not monotone and when we use $r_k$ instead of $r_{[k]}$; compare with Remark \ref{rem:squareBrackets} (e). To this end, consider the CFP with $Q=\mathcal H=\mathbb R^2$, $C_1:=\{(x,y)\colon x\leq 0\}$ and $C_2:=\{(x,y)\colon y\leq 0\}$. Clearly, $C_1 \cap C_2 = (-\infty,0] \times (-\infty,0]$. Define
  \begin{equation}\label{}
    i_k:=
    \begin{cases}
      1, & \mbox{if $k$ is even} \\
      2, & \mbox{otherwise,}
    \end{cases}
    \quad \text{and} \quad
    r_k:=
    \begin{cases}
      \frac 1 {k+1}, & \mbox{if $k$ is even} \\
      \frac 1 {2^k}, & \mbox{otherwise.}
    \end{cases}
  \end{equation}
  Set  $(x_0,y_0):=(1,1)$ and
  \begin{equation}\label{ex:NoFiniteConv:xkyk}
    (x_{k+1},y_{k+1}):= (x_k,y_k)
    + \frac{r_k+d\big((x_k,y_k), C_{i_k}\big)}{2d\big((x_k,y_k), C_{i_k}\big)} \Big(P_{C_{i_k}}\big((x_k,y_k)\big) - (x_k,y_k) \Big).
  \end{equation}
  Then $r_k\to 0$ and $\sum_{k=0}^{\infty}r_k=\infty$, but $(x_k,y_k)\notin C_1\cap C_2$ for all $k=0,1,2,\ldots$. Indeed, observe that $x_k=0$ for all $k=1,2,\ldots$. Moreover, $y_0=y_1=1$ and, by induction,
  \begin{equation}\label{}
    y_{2k-1}=y_{2k-2} \text{\quad and \quad }
    y_{2k}= y_{2k-1} + \frac{\frac 1 {2^{2k-1}}+y_{2k-1}}{2y_{2k-1}} (0- y_{2k-1})
    =\frac 1 {2^{2k}}>0.
  \end{equation}
\end{example}

\begin{example}\label{ex:NoFiniteConv2}
  We show that the subgradient projection method \eqref{ex:phi2:xk}, when combined with a repetitive control (as in \cite{CCP11}), may fail to converge in finitely many steps if we choose to use ``$r_k$'' instead of ``$r_{[k]}$'', even though the sequence of overrelaxations $\{r_k\}_{k=0}^\infty$ is decreasing. Indeed, consider the CFP with $Q=\mathcal H=\mathbb R^2$, $C_1:=\{(x,y)\colon f_1(x,y)\leq 0\}$ and $C_2:=\{(x,y)\colon f_2(x,y)\leq 0\}$, where $f_1(x,y):=|y|-1$ and $f_2(x,y):= x^2-1$. Thus $C_1\cap C_2=[-1,1]\times[-1,1]$ and the Slater condition is satisfied since $f_1(0,0)=f_2(0,0)=-1<0$. Define two auxiliary sequences $\{a_k\}_{k=0}^\infty$ and $\{b_k\}_{k=0}^\infty$ by
  \begin{equation}\label{ex:NoFiniteConv2:akbk}
    a_k:=\frac{1}{k+1} \text{\qquad and\qquad} b_0:=\frac 1 2,
    \quad b_{k+1}:=\frac{b_k}{\left( \frac{2\sqrt 2}{\sqrt{b_k}}+4\right)^2},
  \end{equation}
  and let the sequence of overrelaxations $\{r_k\}_{k=0}^\infty$ consist of all the elements of $\{a_k\}_{k=0}^\infty$ and $\{b_k\}_{k=0}^\infty$ sorted in a decreasing order, that is,
  \begin{multline}
    \{r_k\}_{k=0}^\infty= \left\{a_0=1,\ a_1=\frac 1 2,\ b_0= \frac 1 2,\ a_2=\frac 1 3,\ \ldots\ ,\right.\\
    \left. a_{127}=\frac{1}{128},\ b_1=\frac{1}{128},\ a_{128}=\frac{1}{129},\ \ldots \right\}.
  \end{multline}
  Observe that $r_k\to 0$ monotonically and $\sum_{k=0}^{\infty}r_k=\infty$, as required in \cite{CCP11}. Indeed, the former condition follows from the inequality $b_{k+1}\leq\frac{b_k}{16}$ and the latter one from the definition of $a_k$. Let $m_k$ and $n_k$ denote the position (we start counting from $0$) of $a_k$ and $b_k$ in the sequence $\{r_k\}_{k=0}^\infty$, respectively, and define the control sequence $\{i_k\}_{k=0}^\infty$ by $i_{m_k}:=1$ and $i_{n_k}:=2$. It is not difficult to see that $\{i_k\}_{k=0}^\infty$ is repetitive. Following \eqref{ex:phi2:xk}, we define
  \begin{equation}\label{ex:NoFiniteConv2:xkyk}
    (x_0,y_0):=(2,2),\qquad
    (x_{k+1},y_{k+1}):= (x_{k+1},y_{k+1}) - \frac{r_k+f_{i_k}(x_k,y_k)}{\|g_{i_k}(x_k,y_k)\|^2} g_{i_k}(x_k,y_k)
  \end{equation}
  whenever $f_{i_k}(x_k,y_k)>0$ and $(x_{k+1},y_{k+1}):=(x_k,y_k)$ otherwise, where $g_{i_k}(x_k,y_k) \in \partial f_{i_k}(x_k,y_k)$.
    Then $(x_k,y_k)\notin C_1\cap C_2$ for all $k=0,1,2,\ldots$.
\end{example}

\begin{proof}
  Observe that by \eqref{ex:NoFiniteConv2:xkyk}, we get
  \begin{equation}\label{ex:NoFiniteConv2:x_mk:y_mk}
    x_{m_k+1}=x_{m_k} \text{\qquad and\qquad}
    y_{m_k+1}=
    \begin{cases}
      1-r_{m_k}, & \text{ if } y_{m_k}>1 \\
      r_{m_k}-1, & \text{ if } y_{m_k}<-1 \\
      y_{m_k}, & \text{ otherwise}.\\
    \end{cases}
  \end{equation}
  Moreover,
  \begin{equation}\label{ex:NoFiniteConv2:x_nk:y_nk}
    x_{n_k+1}=\frac{1}{2}\left(x_{n_k}+ \frac{1-r_{n_k}}{x_{n_k}}\right)
    \text{\qquad and\qquad}
    y_{n_k+1}= y_{n_k}.
  \end{equation}
  On the other hand, by the definition of $m_k$ and $n_k$, we have $r_{m_k}=a_k$ and $r_{n_k}=b_k$. Consequently, by the choice of the starting point, we obtain $y_k=0$ for all $k=1,2,\ldots$.

  We claim that $x_{n_k}=1+\sqrt{2b_k}$. Indeed, by the equality $n_0=2$ and by \eqref{ex:NoFiniteConv2:x_mk:y_mk}, we have
  \begin{equation}\label{}
    1+\sqrt{2b_0}=x_0=x_1=x_2=x_{n_0}.
  \end{equation}
  Observe that, by \eqref{ex:NoFiniteConv2:x_mk:y_mk}, we also obtain $x_{n_{k+1}}=x_{n_k+1}$. Consequently, by \eqref{ex:NoFiniteConv2:x_nk:y_nk} and by induction,
  \begin{equation}\label{}
    x_{n_{k+1}}
    =\frac{x_{n_k}^2+1-b_k}{2x_{n_k}}
    =\frac{(1+\sqrt{2b_k})^2+1-b_k}{2(1+\sqrt{2b_k})}
    =1+\frac{b_k}{2+2\sqrt{2b_k}}
    =1+\sqrt{2b_{k+1}}.
  \end{equation}

  Using the positivity of $b_k$, we see that $x_{n_k}>1$ which, when combined with \eqref{ex:NoFiniteConv2:x_mk:y_mk}, yields that $x_k>1$ for all $k=0,1,2,\ldots$. This implies that $(x_k,y_k)\notin C_1\cap C_2$ for all $k=0,1,2,\ldots,$ as claimed.
\end{proof}

\end{document}